\documentclass{amsart}
\usepackage{amsmath,amsthm,latexsym,amsfonts,amssymb,mathrsfs,array,manfnt}
\usepackage[all]{xy}
\usepackage{paralist}
\usepackage{framed}
\usepackage{cancel}
\usepackage{enumitem}
\usepackage{fullpage}

\usepackage{units}
\usepackage{gensymb}
\usepackage{setspace}
\usepackage{amscd}
\usepackage{tikz}
\usetikzlibrary{matrix}

\newcommand{\into}{\hookrightarrow}

\newcommand{\abs}[1]{\left\lvert#1\right\rvert}

\newcommand{\Z}{\ensuremath{\mathbb{Z}}}

\newcommand{\C}{\ensuremath{\mathbb{C}}}
\newcommand{\Q}{\ensuremath{\mathbb{Q}}}

\newcommand{\M}{\mathcal{M}}

\renewcommand{\P}{\ensuremath{\mathbb{P}}}
\renewcommand{\H}{\ensuremath{\mathcal{H}}}

\renewcommand{\bar}[1]{\overline{#1}}

\DeclareMathOperator{\moduli}{moduli}
\DeclareMathOperator{\hurwitz}{hurwitz}

\DeclareMathOperator{\Id}{Id}

\DeclareMathOperator{\Aut}{Aut}

\DeclareMathOperator{\pr}{pr}

\DeclareMathOperator{\PI}{PI}
\usepackage{color}

\newcommand{\cH}{\mathcal{H}}

\newcommand{\cJ}{\mathcal{J}}

\newcommand{\cM}{\mathcal{M}}

\newcommand{\bA}{\mathbf{A}}
\newcommand{\bB}{\mathbf{B}}

\newcommand{\bP}{\mathbf{P}}
\newcommand{\bQ}{\mathbf{Q}}

\newcommand{\bS}{\mathbf{S}}

\theoremstyle{plain}
\newtheorem{theorem}{Theorem}
\numberwithin{theorem}{section}

\newtheorem{thm}[theorem]{Theorem}

\newtheorem{prop}[theorem]{Proposition}

\newtheorem{cor}[theorem]{Corollary}

\newtheorem{lem}[theorem]{Lemma}

\theoremstyle{definition}
\newtheorem{Definition/Theorem}[theorem]{Definition/Theorem}
\newtheorem{Definition/Proposition}[theorem]{Definition/Proposition}

\newtheorem{Def}[theorem]{Definition}

\newtheorem{Corollary/Definition}[theorem]{Corollary/Definition}
\newtheorem{observation}[theorem]{Observation}

\theoremstyle{remark}

\newtheorem{rem}[theorem]{Remark}

\renewcommand{\H}{\cH}
\renewcommand{\M}{\cM}
\newcommand{\Hbar}{\bar{\cH}}
\newcommand{\Mbar}{\bar{\cM}}

\newcommand{\Hfull}{\cH^{\mathrm{full}}}

\newcommand{\Afull}{\bA^{\mathrm{full}}}

\newcommand{\full}{\mathrm{full}}

\renewcommand{\rm}{\mathrm{rm}}

\renewcommand{\setminus}{\smallsetminus}

\DeclareMathOperator{\Th}{Th}

\DeclareMathOperator{\Pn}{\C\P^{\abs{\bP}-3}}
\DeclareMathOperator{\new}{new}
\usetikzlibrary{arrows}
\usetikzlibrary{calc}
\newcommand{\rat}{
\tikz[minimum height=2ex]
  \path[dashed,->]
   node (a)            {}
   node (b) at (1em,0) {}
  ($(a.center)+(0,-.07)$) edge ($(b.center)+(0,-.07)$);
}
\newcommand{\ratt}{
\tikz[minimum height=2ex]
  \path[dashed,->]
   node (a)            {}
   node (b) at (1em,0) {}
  ($(a.center)+(0,.0)$) edge ($(b.center)+(0,.0)$)
  ($(a.center)+(0,-.14)$) edge ($(b.center)+(0,-.14)$);
}
\newcommand{\multi}{
\tikz[minimum height=2ex]
  \path[->]
   node (a)            {}
   node (b) at (1em,0) {}
  ($(a.center)+(0,.0)$) edge ($(b.center)+(0,.0)$)
  ($(a.center)+(0,-.14)$) edge ($(b.center)+(0,-.14)$);
}

\DeclareMathOperator{\br}{br}

\DeclareMathOperator{\Gr}{Gr}
\newcommand{\Pfull}{\mathbf{P}^{\mathrm{full}}}
\usepackage{hyperref}

\begin{document}

\title{Dynamical degrees of Hurwitz correspondences}
\author{Rohini Ramadas}
\email{ramadas@math.harvard.edu}
\address{Department of Mathematics\\Harvard University\\Cambridge, MA}
\thanks{This work was partially supported by NSF grants
0943832, 1045119, 1068190, and 1703308.}

\begin{abstract}
  Let $\phi$ be a post-critically finite branched covering of a two-sphere. By \cite{Koch2013}, the Thurston pullback map induced by $\phi$ on Teichm\"uller space descends to a multi-valued self-map --- a Hurwitz correspondence $\H_{\phi}$ --- of the moduli
  space $\M_{0,\bP}$. We study the dynamics of Hurwitz
  correspondences via numerical invariants called \emph{dynamical
    degrees}. We show that the sequence of dynamical degrees of $\H_{\phi}$ is always non-increasing, and the behavior of this sequence is constrained by the behavior of $\phi$ at and near points of its post-critical set. 
  \end{abstract}
\maketitle

\section{Introduction}\label{sec:Intro}

Denote by $S^2$ the oriented two-sphere. Suppose $\phi:S^2\to S^2$ is an orientation-preserving branched covering whose post-critical set
$$\bP:=\{\phi^n(x)|\text{ $x$ is a critical point of $\phi$ and $n>0$}\}$$
is finite. Then $\phi$ is called \emph{post-critically finite}. The topological dynamics of $\phi$ induce holomorphic and algebraic dynamical systems:
\begin{enumerate}[label=(\Roman*)]
\item Thurston \cite{DouadyHubbard1993}: a holomorphic, contracting self-map $\Th_{\phi}:\mathcal{T}_{S^2,\bP}\to\mathcal{T}_{S^2,\bP}$ of the Teichm\"uller space of complex structures on $S^2$ punctured at $\bP$. This is known as the \emph{Thurston pullback map}, and it descends to \label{item:Thurst}
\item Koch \cite{Koch2013}: an algebraic, multivalued self-map $\H_{\phi}:\M_{0,\bP}\multi\M_{0,\bP}$ of the moduli space of markings of $\C\P^1$ by $\bP$. Such a multivalued map is called a \emph{Hurwitz correspondence}.  \label{item:HC}
\end{enumerate}
In addition, if 
\begin{enumerate}
\item $\bP$ contains a periodic and fully ramified point $p_0$ of $\phi$, and  \label{item:KC1}
\item either every other critical point of $\phi$ is also periodic or there is exactly one other critical point of $\phi$, \label{item:KC2}
\end{enumerate}
then we also have
\begin{enumerate}[label=(\Roman*)]
\setcounter{enumi}{2}
\item Koch \cite{Koch2013}: a meromorphic, single-valued map $\H_{\phi}^{-1}:\M_{0,\bP}\rat\M_{0,\bP}$. \label{item:SVM}
\end{enumerate}

The branched covering $\phi$ is conjugate, up to homotopy, to a post-critically finite rational map on $\C\P^1$ if and only if $\Th_{\phi}$ has a fixed point. There is a tremendous amount of current research investigating the dynamics of \ref{item:Thurst}. Koch introduced \ref{item:HC} and \ref{item:SVM} as algebraic dynamical systems that `shadow' the holomorphic dynamics of \ref{item:Thurst}. 

Dynamical degrees are numerical invariants associated to algebraic dynamical systems; they measure complexity of iteration. Let $X^{\circ}$ be a smooth quasiprojective variety and $g:X^{\circ}\rat X^{\circ}$ (resp. $g:X^{\circ}\ratt X^{\circ}$) a meromorphic map (resp. a meromorphic multi-valued map). Fix a smooth projective birational model $X$ of $X^{\circ}$ and an ample class $\mathfrak{h}\in H^{1,1}(X)$. The $k$th dynamical degree of $g$ is defined to be the non-negative real number
\begin{align*}
  \Theta_k(g)=\lim_{n\to\infty}\left(((g^n)^*(\mathfrak{h}^k))\cdot(\mathfrak{h}^{\dim
      X-k})\right)^{1/n},
\end{align*}
The above limit exists and is independent of $X$ and
$\mathfrak{h}$ (Dinh and Sibony \cite{DinhSibony2005,DinhSibony2008} and Truong \cite{Truong2015,Truong2016}). The $k$-th dynamical degree of $g$ measures the `asymptotic growth rate of the degrees of codimension-$k$ subvarieties of $X^{\circ}$ under iterates of $g$' --- amazingly, this is a well-defined notion although the degree of a subvariety of $X^{\circ}$ is not well-defined. 

Now, since $\phi:\bP\to\bP$ is a self-map of a finite set, every point eventually maps into a periodic cycle. We define the \emph{polynomiality index} of $\phi$ to be the positive real number
$$\PI(\phi):=\max_{\{p\in\bP, \ell>0 | \phi^{\ell}(p)=p\}}\left(\Pi_{i=0}^{\ell-1}(\text{local degree of $\phi$ at $\phi^i(p)$})\right)^{1/\ell}.$$    
In fact, $\PI(\phi)$ is the maximum, over \emph{all} periodic cycles of $\phi$ on $S^2$, of the geometric mean of the local degrees of $\phi$ at all the points in the cycle. 

\medskip
\noindent \textsc{Theorem \ref{thm:DynamicalDegreesNonincreasing}.} \textit{For} $k=0,\ldots,\abs{\bP}-4$, \textit{we have} $\Theta_k(\H_{\phi})\ge \PI(\phi)\cdot\Theta_{k+1}(\H_{\phi})$.

\medskip

Thus the behavior of the sequence of dynamical degrees of $\H_{\phi}$ is constrained by the behavior of $\phi$ at and near points of $\bP$. Note that $1\le\PI(\phi)\le\deg(\phi)$ always holds. $\PI(\phi)=1$ if and only if no critical point of $\phi$ is periodic, i.e. if every critical point is strictly pre-periodic. $\PI(\phi)=\deg(\phi)$ if and only if $\bP$ contains a point $p_0$ that is fixed by and fully ramified under either $\phi$ or $\phi^2$, i.e. if and only if either $\phi$ or $\phi^2$ is a topological polynomial. 

\begin{cor}\label{cor:Decrease}
$\Theta_k(\H_{\phi})$ decreases as $k$ increases, strictly if $\phi$ has a periodic critical point. 
\end{cor}

The dynamical degrees of $\H_{\phi}$ are algebraic integers \cite{Ramadas2015}. As a parallel to Corollary \ref{cor:Decrease}, the results in \cite{Ramadas2015} show that for $k>0$, the degree over $\Q$ of $\Theta_k$ `likely' decreases as $k$ increases. More precisely, there is an upper bound for the degree over $\Q$ of $\Theta_k$ that decreases as $k$ increases. In spite of the parallel, the methods used in this paper are very different from those used in \cite{Ramadas2015}. 

\subsection{Implications when \texorpdfstring{$\H_{\phi}^{-1}$}{H\_{}phi\^{}(-1)} is single-valued}

Dynamical degrees have been studied primarily in the context of single-valued maps. The topological entropy of a holomorphic single-valued map was found to be equal to the logarithm of its largest dynamical degree (Yomdin \cite{Yomdin1987} and Gromov \cite{Gromov2003}). The topological entropy of a meromorphic single-valued map is bounded from above by the logarithm of its largest dynamical degree (Dinh and Sibony \cite{DinhSibony2005}); equality is conjectured. If $g$ is a single-valued map, either holomorphic or meromorphic, its $0$-th dynamical degree is $1$ and its top dynamical degree is its topological degree. Guedj \cite{Guedj2005} found that a map whose top dynamical degree is its largest has especially good ergodic properties. (See Corollary \ref{cor:Guedj} for the implications in the context of this paper.) 

If $g:\C\P^N\to\C\P^N$ is a holomorphic map given in coordinates by homogeneous polynomials of degree $d$, its $k$-th dynamical degree is $d^k$. Thus $k\mapsto\log(\Theta_k(g))$ is linear with slope $d$, and the top ($N$-th) dynamical degree of $g$ is its largest. If $g:\C\P^N\rat\C\P^N$ or $g:X\rat X$ is a meromorphic map, $k\mapsto\log(\Theta_k(g))$ is known to be concave. Thus the top dynamical degree of $g$ is its largest if and only if $k\mapsto\Theta_k(g)$ is strictly increasing.

Koch and Roeder \cite{KochRoeder2015} studied the dynamical degrees of $\H_{\phi}^{-1}$ in the special case that $\phi$ has exactly two critical points, both periodic. They showed in this case that $\Theta_k(\H_{\phi}^{-1})$ is the absolute value of the largest eigenvalue of the induced pullback action on $H^{k,k}(\Mbar_{0,\bP})$, where $\Mbar_{0,\bP}$ is the Deligne-Mumford compactification of $\M_{0,\bP}$. Koch \cite{Koch2013} studied the maps $\H_{\phi}^{-1}$ as meromorphic self-maps of $\Pn$, another compactification of $\M_{0,\bP}$. Koch found that if (\ref{item:KC1}) and (\ref{item:KC2}) hold, and, in addition, the special point $p_0$ is fixed by $\phi$, i.e. $\phi$ is a topological polynomial, then $\H_{\phi}^{-1}:\Pn\rat\Pn$ is holomorphic, given in coordinates by homogeneous polynomials of degree equal to the topological degree of $\phi$. Thus if $\phi$ is a topological polynomial, $\Theta_k(\H_{\phi})=\deg(\phi)^k$. Koch showed that in this case $\H_{\phi}^{-1}:\Pn\to\Pn$ is also \emph{critically finite}. 

Fix $\phi$ of topological degree $d>1$ such that (\ref{item:KC1}) and (\ref{item:KC2}) hold. If there are two fully ramified points of $\phi$ in periodic cycles, then pick $p_0$ to be one with minimal cycle length. Set $\ell_0$ to be the length of the cycle containing $p_0$; then $\PI(\phi)\ge d^{1/\ell_0}.$ Since $\H_{\phi}^{-1}$ is single-valued, its $0$-th dynamical degree is $1$; the results in \cite{Koch2013} imply that its top dynamical degree/topological degree is $d^{\abs{\bP}-3}$. It follows from definitions that $\Theta_k(\H_{\phi}^{-1})=\Theta_{\abs{\bP}-3-k}(\H_{\phi})$. We obtain from Theorem \ref{thm:DynamicalDegreesNonincreasing}:

\begin{cor} \label{cor:SingleValued} The dynamical degrees of $\H_{\phi}^{-1}$ satisfy
\begin{align*}
(d^{1/\ell_0})^{\abs{\bP}-3}&=(d^{1/\ell_0})^{\abs{\bP}-3}\cdot\Theta_0(\H_{\phi}^{-1})\\&\le(d^{1/\ell_0})^{\abs{\bP}-4}\cdot\Theta_1(\H_{\phi}^{-1})\\
&\le\quad\quad\cdots\\
&\le\Theta_{\bP-3}(\H_{\phi}^{-1})=\deg(\H_{\phi}^{-1})=d^{\abs{\bP}-3}
\end{align*}
\end{cor}

In particular, the topological degree of $\H_{\phi}^{-1}$ is strictly larger than its other dynamical degrees. A direct application of Guedj's results in \cite{Guedj2005} yields:

\begin{cor}\label{cor:Guedj}
There is a unique $\H_{\phi}^{-1}$-invariant measure $\mathfrak{m}_{\phi}$ on $\Pn$ of maximal entropy. The measure $\mathfrak{m}_{\phi}$ is mixing, and all its Lyapunov exponents are bounded from below by
$$\frac{1}{2}\log(\PI(\phi))\ge\frac{1}{2\ell_0}\log(d)>0.$$
Further, for any $n>0$, the set of repelling periodic points of $\H_{\phi}^{-1}$ of order $n$ is equidistributed with respect to $\mathfrak{m}_{\phi}$.
\end{cor}

If $\ell_0=1$, then $p_0$ is fixed and $\phi$ is a topological polynomial. In this case, Corollary \ref{cor:SingleValued} recovers that $\Theta_k(\H_{\phi}^{-1})=d^k$, since by \cite{Koch2013}, $\H_{\phi}^{-1}$ is holomorphic on $\Pn$. Thus in this case $k\mapsto \log(\Theta_k(\H_{\phi}^{-1}))$ is linear of slope $\log(d)$. If $\ell_0>1$, then $k\mapsto \log(\Theta_k(\H_{\phi}^{-1}))$ is a concave function, which by Corollary \ref{cor:SingleValued} is strictly increasing with slope at least $(1/\ell_0)\log(d)$. This generalizes the result that polynomiality of $\phi$ ensures holomorphicity of $\H_{\phi}^{-1}$ on $\Pn$ as follows:

\begin{observation}
The more $\phi$ resembles a topological polynomial, i.e. the smaller the value of $\ell_0$, the more the sequence of dynamical degrees of $\H_{\phi}^{-1}$ resembles the sequence of dynamical degrees of a holomorphic map on $\Pn$ (Figure \ref{fig:Graphs}).
\end{observation}

\begin{figure}
\begin{center}
 \begin{tikzpicture}[scale=.6]
    \draw[->] (-.2,0) -- (10+2,0);
    \draw[->] (0,-.2) -- (0,6.9897+2);
    \draw [red] (0,0)--(10,6.9897);
    \draw (0,3.4949)--(10,6.9897);
    \draw (0,4.6598)--(10,6.9897);
    \draw (0,5.2423)--(10,6.9897);
    \draw (0,5.5918)--(10,6.9897);
    \draw (10,6.9897) node[right] {\scriptsize $(\abs{\bP}-3, d(\abs{\bP}-3))$};
    \draw (0,0) node[left] {\scriptsize $\ell_0=1$};
    \draw (0,3.4949) node[left] {\scriptsize $\ell_0=2$};
    \draw (0,4.6598) node[left] {\scriptsize$\ell_0=3$};
    \draw (0,5.2423) node[left] {\scriptsize$\ell_0=4$};
    \draw (0,5.6918) node[left] {\scriptsize$\text{ slope}=(1/\ell_0)d\quad\quad\ell_0=5$};
    \draw (1,-.1)--(1,.1);
    \draw (1,-.4) node {\scriptsize1};
    \draw (2,-.1)--(2,.1);
    \draw (2,-.4) node {\scriptsize2};
    \draw (5,-.4) node {$\cdots$};
    \draw (8,-.1)--(8,.1);
    \draw (9,-.1)--(9,.1);
    \draw[thick] (10,-.1)--(10,.3);
    \draw (10,-.4) node {\scriptsize$\abs{\bP}-3$};
    \draw (12,-.4) node {\scriptsize$k$};
    \draw (-1,9.4897) node {\scriptsize$\log(\delta_k(\mathcal{H}_\phi^{-1}))$};
    \draw[dashed,red] (0,0) to[out=65,in=199.265] (10,6.9897);
  \end{tikzpicture}
\caption{Fix $d>1$ and $\abs{\bP}\ge3$. Given a degree $d$ finite branched covering $\phi$ whose post-critical set has size $\abs{\bP}$ such that there exists a fully ramified point in a periodic cycle of length $\ell_0$, the figure shows how the graph of $k\mapsto\log(\Theta_k(\H_{\phi}^{-1}))$ is constrained by $\ell_0$. If $\ell_0=1$, the graph is the line of slope $\log(d)$, pictured in solid red. If $\ell_0>1$, the graph is concave of slope at least $(1/\ell_0)\log(d)$, passing through $(0,0)$ and $(\abs{P}-3, d(\abs{P}-3))$, and between the line of slope $(1/\ell_0)\log(d)$ pictured in solid black and the line of slope $\log(d)$ pictured in solid red. The dashed red curve depicts qualitatively what the graph might look like if $\ell_0=2$.}
\end{center}
\label{fig:Graphs}
\end{figure}
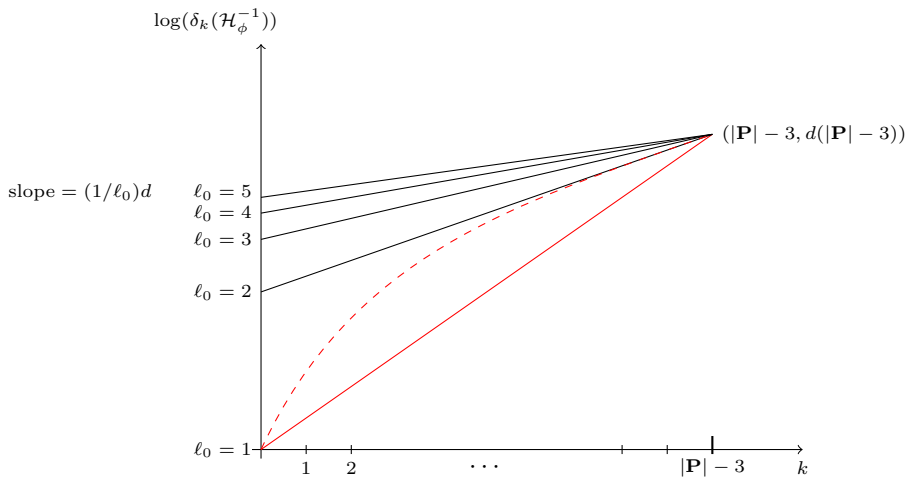
\subsection{Implications for Hurwitz correspondences as multi-valued maps and an application to enumerative geometry}

Hurwitz correspondences can be defined without reference to the Thurston pullback map. Let $\bP$ be a finite set and let $\H$ be a Hurwitz space
parametrizing maps $f:\C\P^1\to \C\P^1$ together with two injections from $\bP$ into the source and target $\C\P^1$ respectively, such that $f$ has specified branching behavior at and over the marked points $\bP$. The Hurwitz space $\H$ admits two maps to $\M_{0,\bP}$: a 
map $\pi_1$ specifying the configuration of marked points on the ``target'' $\C\P^1$, and a map $\pi_2$ specifying the configuration of marked points on the ``source" $\C\P^1$. If the marked points on the target $\C\P^1$ include all
the branch values of $f$, then $\pi_1$ is a covering map, and
$\pi_2\circ\pi_1^{-1}$ defines a multi-valued map from $\M_{0,\bP}$ to itself. 

The Hurwitz space $\H$ may be disconnected; each connected component parametrizes maps of a single topological type. If $\phi$ is a post-critically finite branched covering with post-critical set $\bP$ and branching type as specified by $\H$, then $\H_{\phi}$ is the connected component of $\H$ that parametrizes maps $f:\C\P^1\to\C\P^1$ such that there exist homeomorphisms $\chi_1$ and $\chi_2$ from $(\C\P^1,\bP)$ to $(S^2, \bP)$ with $\chi_2\circ f= \phi\circ\chi_1$. Every connected component of $\H$ arises as $\H_{\phi}$ for some post-critically finite $\phi$. Fixing $\H$, all such branched coverings have the same branching, induce the same map $\bP\to\bP$, and in particular, have the same polynomiality index. Thus polynomiality index is a well-defined invariant of $\H$. We define a Hurwitz correspondence in general to be the multivalued self-map of $\M_{0,\bP}$ obtained by restricting $\pi_2\circ\pi_1^{-1}$ to any non-empty union $\Gamma$ of connected components of a Hurwitz space $\H$. Theorem \ref{thm:DynamicalDegreesNonincreasing} is proved in this more general context, i.e. we have $\Theta_k(\Gamma)\ge \PI(\Gamma)\cdot\Theta_{k+1}(\Gamma)$. This implies that Hurwitz correspondences are a special subclass of multivalued maps: The sequence of dynamical degrees of a multi-valued map may not be log-concave, and appears to be quite unconstrained in general \cite{Truong2016}. 

The topological entropy of a multi-valued map is at most the logarithm of its largest dynamical degree, but can be strictly smaller \cite{DinhSibony2008}. Thus we have:

\begin{cor}
The topological entropy of $\H$ (resp. $\H_{\phi}$ or $\Gamma$) is at most its $0$-th dynamical degree $\Theta_0(\H)$ (resp. $\Theta_0(\H_{\phi})$ or $\Theta_0(\Gamma)$).
\end{cor}

The $0$-th dynamical degree of $\H$ is the topological degree of the ``target" map $\pi_1:\H\to\M_{0,\bP}$. This degree is called a \emph{Hurwitz number}; it counts covers of $\C\P^1$ having specified branch locus on $\C\P^1$ and specified ramification profile. It also counts the number of ways to factor the identity in
the symmetric group $S_d$ as a product of permutations with specified cycle types that
collectively generate a transitive subgroup. Thus the dynamically
motivated quantity $\Theta_0(\H)$ has a purely combinatorial
interpretation.

The top dynamical degree of $\H$ is the topological degree of the ``source" map $\pi_2$. In Section \ref{sec:enumerative}, we use Theorem \ref{thm:DynamicalDegreesNonincreasing} to prove:

\medskip

\noindent\textsc{Proposition \ref{prop:enumerative}.} \textit{Let $r$ be the maximum local degree of $[f:\C\P^1\to\C\P^1]\in\H$ at $p$ where $p$ ranges over $\bP$. Then $\deg(\pi_1)\ge r^{\abs{\bP}-3} \deg(\pi_2)$. That is, the number of ways a generic configuration of $\bP$-marked points on $\C\P^1$ arises the configuration of marked points on the target $\C\P^1$ of some branched map $[f]\in\H$ is at least $r^{\abs{\bP}-3}$ times the number of ways it appears as the configuration of marked points on the source $\C\P^1$ of such a map.}

\subsection{Organization}
Section \ref{sec:Background} gives background on meromorphic multi-valued maps (henceforth referred to as rational correspondences), the moduli space
$\M_{0,\bP}$ and its compactification $\Mbar_{0,\bP}$, Hurwitz spaces,
and Hurwitz correspondences. Section \ref{sec:Proof} contains the
proof of Theorem \ref{thm:DynamicalDegreesNonincreasing}. Section \ref{sec:enumerative} contains an application of the main theorem to enumerative algebraic geometry.

\subsection{Acknowledgements}
  I am grateful to my advisors David Speyer and  Sarah
  Koch for introducing me to Hurwitz correspondences, and for
  wonderful guidance. I am grateful to Mattias Jonsson, Philip Engel,
  Karen Smith, Roland Roeder and Rob Silversmith for useful conversations, and to
  Rob Silversmith for help with typing. I am also grateful to an anonymous referee for useful comments on an earlier version.
  
\subsection{Conventions}
All varieties are over $\C$. For $X$ a variety, we denote
by $\mathcal{Z}_k(X)$ the group of $k$-cycles on $X$; that is, the
free abelian group on the set of $k$-dimensional subvarieties of
$X$. We denote by $A_k(X)$ the Chow group of $k$-cycles on $X$ up to
rational equivalence. For $X$ a smooth variety, we denote by $A^k(X)$
the Chow group of codimension-$k$ cycles on $X$.
\section{Background}\label{sec:Background}
\subsection{Rational
  correspondences/meromorphic multi-valued maps}\label{sec:RationalCorrespondences}
Rational correspondences generalize the notion of a rational map. A
rational correspondence from $X$ to $Y$ is a multi-valued map to $Y$
defined on a dense open set of $X.$
\begin{Def}
  Let $X$ and $Y$ be irreducible smooth projective varieties. A
  \emph{rational correspondence} $(\Gamma,\pi_X,\pi_Y):X\ratt
  Y$ is a diagram
  \begin{center}
    \begin{tikzpicture}
      \matrix(m)[matrix of math nodes,row sep=3em,column
      sep=4em,minimum width=2em] {
        &\Gamma&\\
        X&&Y\\}; \path[-stealth] (m-1-2) edge node [above left]
      {$\pi_X$} (m-2-1); \path[-stealth] (m-1-2) edge node
      [above right]
      {$\pi_Y$} (m-2-3);
    \end{tikzpicture}
  \end{center}
  where $\Gamma$ is a smooth quasiprojective variety, not necessarily
  irreducible, and the restriction of $\pi_X$ to every irreducible
  component of $\Gamma$ is dominant and generically finite.
\end{Def}
Over a dense open set in $X$, $\pi_X$ is a covering map, and
$\pi_Y\circ\pi_X^{-1}$ defines a multi-valued map to $Y.$ However,
considered as a multi-valued map from $X$ to $Y,$ it is possible
that $\pi_Y\circ\pi_X^{-1}$ has indeterminacy, since some fibers of
$\pi_X$ may be empty or positive-dimensional.

Like rational maps, rational correspondences induce pushforward and
pullback maps of Chow groups, and can be composed with each other.
\begin{Def}\label{Def:Cycle}
  Let $\bar{\Gamma}$ be a projective compactification of $\Gamma$ such
  that $\Gamma$ is dense in $\bar{\Gamma}$ and $\pi_X$ and $\pi_Y$
  extend to maps $\bar{\pi_X}$ and $\bar{\pi_Y}$ defined on
  $\bar{\Gamma}$. The cycle
  $(\bar{\pi_X}\times\bar{\pi_Y})_*[\bar{\Gamma}]\in\mathcal{Z}_{\dim
    X}(X\times Y)$ is independent of the choice of
  compactification $\bar{\Gamma}$, so we denote this cycle by
  $[\Gamma]$.
\end{Def}
\begin{rem}
  In \cite{DinhSibony2008}, a rational correspondence from $X$ to
  $Y$ is defined as a cycle $\sum_im_i[\Gamma_i]$ in
  $\mathcal{Z}_{\dim X}(X\times Y),$ such that each $\Gamma_i$
  maps surjectively onto $X.$
\end{rem}
\begin{Def}
  Let $\bar{\Gamma}$ be a projective compactification of $\Gamma$ as
  in Definition \ref{Def:Cycle}. Set
  \begin{align*}
    [\Gamma]_*:&=(\bar{\pi_Y})_*\circ(\bar{\pi_X})^*:A_k(X)\to
    A_k(Y)
  \end{align*}
  and
  \begin{align*}
    [\Gamma]^*:&=(\bar{\pi_X})_*\circ(\bar{\pi_Y})^*:A^k(Y)\to
    A^k(X).
  \end{align*}
\end{Def}
These pushforward and pullback maps are independent of the choice of
compactification $\bar{\Gamma}$; they depend only on the cycle
$[\Gamma]$ (\cite{Fulton1998}, Remark 6.2.2).
\begin{Def}\label{Def:Composite}
  Suppose $(\Gamma,\pi_X,\pi_Y):X\ratt Y$ and
  $(\Gamma',\pi_Y',\pi_Z'):Y\ratt Z$ are rational correspondences such
  that the image under $\pi_Y$ of every irreducible component of
  $\Gamma$ intersects the domain of definition of the multi-valued map
  $\pi_Z'\circ(\pi_Y')^{-1}.$ The \emph{composite}
  $\Gamma'\circ\Gamma$ is a rational correspondence from $X$ to $Z$
  defined as follows.

  Pick dense open sets $U_X\subseteq X$ and $U_Y\subseteq Y$ such
  that $\pi_Y(\pi_X^{-1}(U_X))\subseteq U_Y,$ and
  $\pi_X|_{\pi_X^{-1}(U_X)}$ and $\pi_Y'|_{(\pi_Y')^{-1}(U_Y)}$ are
  both covering maps. Set
  $$\Gamma'\circ
  \Gamma:=\pi_X^{-1}(U_X)\thickspace\thickspace{_{\pi_Y}\times_{\pi_Y'}}\thickspace\thickspace(\pi_Y')^{-1}(U_Y),$$
  together with its given maps to $X$ and $Z$.
\end{Def}
This composite does depend on the choices of open sets $U_X$ and
$U_Y$, but the cycle $[\Gamma'\circ\Gamma]$ is well-defined. Note
that $[\Gamma'\circ\Gamma]_*$ may not agree with
$[\Gamma']_*\circ[\Gamma]_*$ and $[\Gamma'\circ\Gamma]^*$ may not
agree with $[\Gamma]^*\circ[\Gamma']^*$.
\subsection{Dynamical degrees}\label{sec:DynamicalDegrees}
Dynamical degrees were first introduced as invariants of surjective
holomorphic self-maps of a smooth projective variety. The $k$th
dynamical degree of $g:X\to X$ is the spectral radius of
$g^*:H^{k,k}(X)\to H^{k,k}(X).$ Dynamical degrees were later generalized to
rational maps and rational correspondences.
\begin{Def}\label{Def:Dominant}
  Let $(\Gamma,\pi_1,\pi_2):X\ratt X$ be a rational
  correspondence such that the restriction of $\pi_2$ to every
  irreducible component of $\Gamma$ is dominant.  In this case we say
  $\Gamma$ is a \emph{dominant} rational self-correspondence. 
\end{Def}
\begin{Def}
  Let $\Gamma$ be as in Definition \ref{Def:Dominant}. Set
  $\Gamma^n:=\Gamma\circ\cdots\circ\Gamma$ ($n$ times), and pick
  $\mathfrak{h}$ an ample divisor class on $X$. The \emph{$k$th
    dynamical degree} $\Theta_k$ of $\Gamma$ is defined to be
  $$\lim_{n\to\infty}\left(([\Gamma^n]^*(\mathfrak{h}^k))\cdot(\mathfrak{h}^{\dim
      X-k})\right)^{1/n}.$$ This limit exists and is independent of
  choice of ample divisor
  (\cite{DinhSibony2005,DinhSibony2008,Truong2015,Truong2016}).
\end{Def}
The dynamical degrees of $\Gamma$ are determined by the cycle
$[\Gamma]$.
\begin{thm}[(Birational invariance of dynamical
  degrees, \cite{DinhSibony2005,DinhSibony2008,Truong2015,Truong2016})]\label{thm:BirationalInvariance}
  Let $(\Gamma,\pi_1,\pi_2):X\ratt X$ be a dominant rational
  self-correspondence, and let $\beta:X\rat X'$ be a birational
  equivalence. We obtain a dominant rational self-correspondence on
  $X'$ through conjugation by $\beta$ as follows. Let $U$ be the domain
  of definition of $\beta$, and set
  $\Gamma'=\pi_1^{-1}(U)\cap\pi_2^{-1}(U).$ We have a dominant
  rational self-correspondence
\begin{align*}
  (\Gamma',\beta\circ\pi_1,\beta\circ\pi_2):X'\ratt X'.
\end{align*}
Then the dynamical degrees of $\Gamma$ and $\Gamma'$ are equal.
\end{thm}
\noindent\textbf{The sequence of dynamical degrees of a rational map
  is log-concave.} Let $g:X\rat X$ be a dominant rational map, and let
$\mathfrak{h}$ be an ample divisor class on $X$. For $n>0$ set
$\Gr(g^n)$ to be the graph of $g^n$ in $X\times X$, with its two maps
$\pi_1^n$ and $\pi_2^n$ to $X$. If $\Theta_k$ denotes the $k$th
dynamical degree of $g$, we
have
\begin{align*}
  \Theta_k&=\lim_{n\to\infty}\left(((g^n)^*(\mathfrak{h}^k))\cdot(\mathfrak{h}^{\dim
      X-k})\right)^{1/n}\\
  &=\lim_{n\to\infty}\left(((\pi_2^n)^*(\mathfrak{h}^k))\cdot((\pi_1^n)^*(\mathfrak{h}^{\dim
      X-k}))\right)^{1/n},
\end{align*}
by the projection formula. Since $(\pi_2^n)^*(\mathfrak{h})$ and
$(\pi_1^n)^*(\mathfrak{h})$ are nef on $\Gr(g^n),$ and $\Gr(g^n)$
is irreducible, the sequence of intersection numbers
$\{((\pi_2^n)^*(\mathfrak{h}^k))\cdot((\pi_1^n)^*(\mathfrak{h}^{\dim X-k}))\}_k$
is log-concave (\cite{Lazarsfeld2004}, Example 1.6.4). Thus the
sequence $\{\Theta_k\}_k$ is log-concave as well.

This statement is false for multi-valued maps/rational
correspondences \cite{Truong2016}. The argument breaks down since their graphs are not
necessarily irreducible. Even if a given rational correspondence $\Gamma$ is irreducible, its iterates $\Gamma^{n}$ are reducible in general. Our proof of
Theorem \ref{thm:DynamicalDegreesNonincreasing} deals separately with
every irreducible component of infinitely many iterates of a given Hurwitz
correspondence.
\subsection{The moduli spaces \texorpdfstring{$\M_{0,\bP}$}{M\_{}\{0,P\}} and \texorpdfstring{$\Mbar_{0,\bP}$}{Mbar\_{}\{0,P\}}}
The moduli space $\M_{0,\bP}$ is a smooth quasiprojective
variety parametrizing ways of marking $\C\P^1$ by elements of a finite
set, up to change of coordinates on $\C\P^1$.
\begin{Def}
  Let $\abs{\bP}\ge3.$ There is a smooth quasiprojective variety
  $\M_{0,\bP}$ of dimension $\abs{\bP}-3$ parametrizing injections $\iota:\bP\into\C\P^1$ up to post-composition by M\"obius transformations of $\C\P^1$ 
\end{Def}
There are several compactifications of $\M_{0,\bP}$ that extend the interpretation as a moduli space. The most widely-studied of these is the Deligne-Mumford/
stable curves compactification $\Mbar_{0,\bP}$. Projective space $\C\P^{\abs{\bP}-3}$ is another such compactification.
\begin{Def}\label{Def:StableCurve}
  A \emph{stable $\bP$-marked genus zero curve} is a connected
  projective curve $C$ of arithmetic genus zero whose only
  singularities are simple nodes, together with an injection
  $\iota:\bP\into(\mbox{smooth locus of $C$})$, such that the set of
  automorphisms $C\to C$ that commute with $\iota$ is finite.
\end{Def}
\begin{thm}[(Deligne, Grothendieck, Knudsen, Mumford)]
  There is a smooth projective variety $\Mbar_{0,\bP}$ of dimension
  $\abs{\bP}-3$ that parametrizes stable $\bP$-marked genus zero
  curves. It contains $\M_{0,\bP}$ as a dense open subset.
\end{thm}
The complement $\Mbar_{0,\bP}\setminus\M_{0,\bP}$ is a simple normal
crossings divisor, referred to as the \emph{boundary} of
$\Mbar_{0,\bP}$. Given a subset $\bS\subseteq\bP$ such that
$\abs{\bS},\abs{\bS^C}\ge2,$ define a divisor
$\delta_\bS\subseteq\Mbar_{0,\bP}$ as follows. Consider the locus of
all $[C,\iota]$ in $\Mbar_{0,\bP}$ such that $C$ has two irreducible
components joined at a node, the points $\iota(p)$ with $p\in\bS$ are
all on one component, and the points $\iota(p)$ with $p\in\bS^C$ are
all on the other component. Let $\delta_\bS$ be the closure of this
locus; $\delta_\bS$ is an irreducible divisor contained in the
boundary. Every irreducible component of the boundary is obtained in
this manner. Note that $\delta_\bS=\delta_{\bS^C}.$
\begin{Def}
  For an injection $j:\bP'\into\bP$ with $\abs{\bP'}\ge3$, there is a
  \emph{forgetful map} $\mu:\M_{0,\bP}\to\M_{0,\bP'}$ sending
  $[C,\iota]$ to $[C,\iota\circ j].$ This map extends to
  $\mu:\Mbar_{0,\bP}\to\Mbar_{0,\bP'}$.
\end{Def}

\noindent\textbf{The tautological $\psi$-classes.}
$\Mbar_{0,\bP}$ has a tautological line bundle $\mathcal{L}_p$
corresponding to each marked point $p\in\bP.$ This line bundle assigns
to the point $[C,\iota]$ the 1-dimensional complex vector space
$T_{\iota(p)}^\vee C,$ namely, the cotangent line to the curve $C$ at
the marked point $\iota(p)$. The divisor class associated to
$\mathcal{L}_p$ is denoted $\psi_p$.

The space $H^0(\Mbar_{0,\bP},\mathcal{L}_p)$ is
$(\abs{\bP}-2)$-dimensional and basepoint-free. The induced map
$\rho:\Mbar_{0,\bP}\to\P(H^0(\Mbar_{0,\bP},\mathcal{L}_p)^\vee)\cong\C\P^{\abs{\bP}-3}$
is a birational map onto $\C\P^{\abs{\bP}-3}$ (\cite{Kapranov1993}).

\medskip

Consider a forgetful map
$\mu:\Mbar_{0,\bP\cup\{q\}}\to\Mbar_{0,\bP}$. For $p\in\bP,$ we have (\cite{ArbarelloCornalba1998})
\begin{align*}
  \mu^*\psi_p^{\Mbar_{0,\bP}}=\psi_p^{\Mbar_{0,\bP\cup\{q\}}}-\delta_{\{p,q\}}.
\end{align*}
It follows by induction that:
\begin{lem}\label{lem:ForgetPsi}
  For a forgetful map $\mu:\Mbar_{0,\bP\cup\bQ}\to\Mbar_{0,\bP}$, we
  have
  \begin{align*}
    \mu^*\psi_p^{\Mbar_{0,\bP}}=\psi_p^{\Mbar_{0,\bP\cup\bQ}}-\sum_{\substack{\bS\subseteq\bQ\\\text{$\bS$
          nonempty}}}\delta_{\{p\}\sqcup \bS}.
  \end{align*}
\end{lem}
\subsection{Hurwitz spaces and Hurwitz correspondences}\label{sec:HurwitzSpaces}
Hurwitz spaces are moduli spaces parametrizing finite maps with
prescribed ramification between smooth algebraic curves/Riemann surfaces. See
\cite{RomagnyWewers2006} for a summary.
\begin{Def}
  A \emph{partition} $\lambda$ of a positive integer $k$ is a
  multiset of positive integers whose sum with multiplicity is $k$.
\end{Def}
\begin{Def}
  A multiset $\lambda_1$ is a \emph{submultiset} of $\lambda_2$ if
  for all $r\in\lambda_1,$ the multiplicity of occurrence of $r$ in
  $\lambda_1$ is less than or equal to the multiplicity of occurrence
  of $r$ in $\lambda_2$.
\end{Def}
\begin{Def}[(\emph{Hurwitz space}, \cite{Ramadas2015}, Definition 5.4)]\label{Def:HurwitzSpace}
  Fix discrete data:
  \begin{itemize}
  \item $\bA$ and $\bB$ finite sets with cardinality at least 3
    (marked points on source and target curves, respectively),
  \item $d$ a positive integer (degree),
  \item $F:\bA\to \bB$ a map,
  \item $\br:\bB\to\{\mbox{partitions of $d$}\}$ (branching), and
  \item $\rm:\bA\to\Z^{>0}$ (ramification),
  \end{itemize}
  such that
  \begin{itemize}
  \item (Condition 1, Riemann-Hurwitz constraint) $\sum_{b\in\bB}\left(d-\mbox{length of
        $\br(b)$}\right)=2d-2$, and
  \item (Condition 2) for all $b\in\bB,$ the multiset $(\rm(a))_{a\in
      F^{-1}(b)}$ is a submultiset of $\br(b)$.
  \end{itemize}  
  There exists a smooth quasiprojective variety
  $\cH=\cH(\bA,\bB,d,F,\br,\rm),$ a \emph{Hurwitz space},
  parametrizing morphisms $f:\C\P^1\to \C\P^1$ up to isomorphism, where
  \begin{itemize}
  \item There are injections from $\bA$ and $\bB$ into the source and target $\C\P^1$ respectively,
  \item $f$ is degree $d$,
  \item for all $a\in\bA,$ $f(a)=F(a)$ via the injections of $\bA$ and $\bB$ into $\C\P^1$, 
  \item for all $b\in\bB,$ the branching of $f$ over $b$ is given by
    the partition $\br(b)$, and
  \item for all $a\in\bA,$ the local degree of $f$ at $a$ is equal to
    $\rm(a)$.  
  \end{itemize}
\end{Def}
The Hurwitz space $\cH$ has a ``source'' map $\pi_{\bA}$ to
$\M_{0,\bA}$ sending $[f:(\C\P^1, \bA)\to(\C\P^1,\bB)]$ to $[\C\P^1,\bA]$. There is
similarly a ``target'' map $\pi_{\bB}$ from $\H$ to
$\M_{0,\bB}$. Unless $\cH$ is empty, $\pi_{\bB}$ is a finite covering
map. Thus for smooth compactifications $X_{\bA}$ of $\M_{0,\bA}$ and
$X_{\bB}$ of $\M_{0,\bB}$, $(\cH,\pi_{\bB},\pi_{\bA}):X_{\bB}\ratt
X_{\bA}$ is a rational correspondence. We generalize this notion.
\begin{Def}[(\emph{Hurwitz correspondence}, \cite{Ramadas2015},
  Definition 5.5)]
  Let $\bA'$ be any subset of $\bA$ with cardinality at least 3. There
  is a forgetful morphism $\mu:\M_{0,\bA}\to\M_{0,\bA'}$. Let $\Gamma$
  be a union of connected components of $\cH$. If $X_{\bA'}$ and
  $X_{\bB}$ are smooth projective compactifications of $\M_{0,\bA'}$
  and $\M_{0,\bB}$ respectively,
  then $$\left(\Gamma,\pi_{\bB},\mu\circ\pi_{\bA}\right):X_{\bB}\ratt
  X_{\bA'}$$ is a rational correspondence. We call such a rational
  correspondence a \emph{Hurwitz correspondence}.
\end{Def}

\subsection{Hurwitz self-correspondences and dynamics}

Suppose $\phi:S^2\to S^2$ is a degree $d$ orientation-preserving branched covering with finite post-critical set $\bP$. Define $\br:\bP\to\{\text{partitions of $d$}\}$ sending $p\in\bP$ to the branching profile of $\phi$ over $p$. Define $\rm:\bP\to\Z^{>0}$ sending $p\in\bP$ to the local degree of $\phi$ at $p$. Then
\begin{align*}
\H=\H(\bP,\bP,d,\phi|_{\bP},\br,\rm)
\end{align*}
parametrizes maps $(\C\P^1,\bP)\to(\C\P^1,\bP)$ with the same branching as $\phi$. Let $\pi_1$ and $\pi_2$ be the ``target" and ``source" maps from $\H$ to $\M_{0,\bP}$. For $\Gamma$ a non-empty union of connected
components of $\H$, and $X_{\bP}$ any compactification of $\M_{0,\bP}$, $(\Gamma,\pi_1,\pi_2):X_{\bP}\ratt X_{\bP}$ is a rational self-correspondence.

There is a unique connected component $\H_{\phi}$ of $\H$ parametrizing maps that are topologically isomorphic to $\phi$, i.e. maps $f:(\C\P^1,\bP)\to(\C\P^1,\bP)$ such that there exist marked-point-preserving homeomorphisms $\chi_1$ and $\chi_2$ from $(\C\P^1,\bP)$ to $(S^2, \bP)$ with $\chi_2\circ f= \phi\circ\chi_1$. By \cite{Koch2013}, the multi-valued map defined by $\H_{\phi}$ on $\M_{0,\bP}$ is descended from the Thurston pullback map $\Th_{\phi}$. 

It is convenient to consider Hurwitz self-correspondences in more generality. Given a Hurwitz space $\H=\H(\bP',\bP,d,F,\br,\rm)$ together with an injection $\bP\into\bP',$ if $\mu:\M_{0,\bP'}\to\M_{0,\bP}$ is the forgetful map, $\Gamma$ is a non-empty union of connected
components of $\H$, and $X_{\bP}$ is a compactification of $\M_{0,\bP}$, then $(\Gamma,\pi_{\bP},\mu\circ\pi_{\bP'}):X_{\bP}\ratt X_{\bP}$ is a Hurwitz
self-correspondence. Note that by Theorem \ref{thm:BirationalInvariance}, the
dynamical degrees of the Hurwitz self-correspondence $\Gamma$ do not
depend on the choice of compactification $X_{\bP}$.

\begin{Def}\label{def:PI}
As above, let $\H=\H(\bP',\bP,d,F,\br,\rm)$ be a Hurwitz space together with an injection $\bP\into\bP'$. Since $F:\bP\to\bP$ is a self-map of a finite set, every point eventually maps into a periodic cycle. We define the \emph{polynomiality index} of $\H$ to be 
$$\PI(\H):=\max_{\{p\in\bP, \ell>0 | F^{\ell}(p)=p\}}\left(\Pi_{i=0}^{\ell-1}\rm(p)\right)^{1/\ell}.$$ 
If $\Gamma$ is a non-empty union of connected components of $\H$, then we define the polynomiality index of $\Gamma$ to be the polynomiality index of $\H$. 
\end{Def}

Note that the polynomiality index of $\H_{\phi}$ as in Definition \ref{def:PI} agrees with the polynomiality index of $\phi$ as in Section \ref{sec:Intro}.
\subsection{Fully marked Hurwitz spaces and admissible covers}
Harris and Mumford (\cite{HarrisMumford1982}) constructed
compactifications of Hurwitz spaces. These compactifications are
called moduli spaces of \emph{admissible covers}. They are
projective varieties that parametrize certain ramified maps between
nodal curves. They extend the ``target curve'' and ``source curve''
maps to the stable curves compactifications of the moduli spaces of
target and source curves, respectively.

In general, the admissible covers compactifications are only coarse
moduli spaces with \emph{orbifold singularities}. For technical ease,
we introduce a class of Hurwitz spaces whose admissible covers
compactifications are fine moduli spaces. We call these Hurwitz spaces
\emph{fully marked}.
\begin{Def}[(\cite{Ramadas2015}, Definition 5.6)]\label{Def:FullyMarked}
  Given $(\bA,\bB,d,F,\br,\rm)$ as in Definition
  \ref{Def:HurwitzSpace} with Condition 2 strengthened to:
  \begin{itemize}
  \item (Condition $2'$) For all $b\in\bB,$ the multiset
    $(\rm(a))_{a\in F^{-1}(b)}$ is \textbf{equal to} $\br(b),$
  \end{itemize}
  we refer to the corresponding Hurwitz space
  $\cH(\bA,\bB,d,F,\br,\rm)$ as a \emph{fully marked Hurwitz
    space}.
\end{Def}
Given any Hurwitz space $\cH=\cH(\bA,\bB,d,F,\br,\rm)$, there exists a
fully marked Hurwitz space $\Hfull=\H(\Afull,\bB,d,F,\br,\rm)$, where
$\Afull$ is a superset of $\bA$ extending the functions $F$ and
$\rm$. There is a finite covering map $\nu:\Hfull\to\cH$, and we have
the following commutative diagram (see \cite{Ramadas2015} for
details):
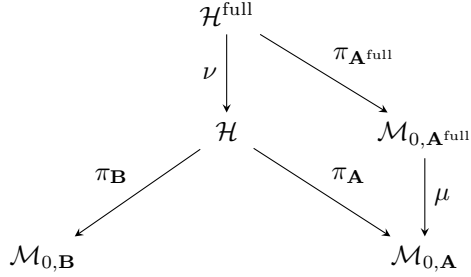
\begin{center}
  \begin{tikzpicture}
    \matrix(m)[matrix of math nodes,row sep=3em,column
    sep=4em,minimum width=2em] {
      &\Hfull&\\
      &\H&\M_{0,\Afull}\\
      \M_{0,\bB}&&\M_{0,\bA}\\};
    \path[-stealth] (m-1-2) edge node [left] {$\nu$} (m-2-2);
    \path[-stealth] (m-1-2) edge node [above right] {$\pi_{\Afull}$}
    (m-2-3);
    \path[-stealth] (m-2-3) edge node [right] {$\mu$} (m-3-3);
    \path[-stealth] (m-2-2) edge node [above right] {$\pi_{\bA}$} (m-3-3);
    \path[-stealth] (m-2-2) edge node [above left] {$\pi_{\bB}$} (m-3-1);
  \end{tikzpicture}
\end{center}
For $\Gamma$ a union of connected components of $\H,$ and for
$X_{\bB}$ and $X_{\bA}$ smooth projective compactifications of
$\M_{0,\bB}$ and $\M_{0,\bA}$, respectively,
$(\Gamma,\pi_\bB,\pi_\bA):X_\bB\ratt X_\bA$ is a Hurwitz
correspondence. Set $\Gamma^{\full}=\nu^{-1}(\Gamma).$ Then
$\Gamma^{\full}$ is a union of connected components of $\Hfull,$ and
in $\mathcal{Z}_{\dim X_\bB}(X_\bB\times
X_\bA),$ $$[\Gamma]=\frac{1}{\deg\nu}[\Gamma^{\full}].$$
\begin{lem}\label{lem:ReduceToConnectedFullyMarked}
  Let $(\Gamma,\pi_1,\pi_2):X_\bP\ratt X_\bP$ be a dominant Hurwitz
  self-correspondence. Then
  \begin{align*}
    (\mbox{$k$th dynamical degree of $\Gamma$})=\frac{1}{\deg\nu}(\mbox{$k$th dynamical degree of $\Gamma^{\full}$}),
  \end{align*}
  where $\Gamma^{\full}$ is a union of connected components of a fully
  marked Hurwitz space $\Hfull$ corresponding to a superset $\Pfull$
  of $\bP,$ and $\nu:\Gamma^{\full}\to\Gamma$ is a finite covering
  map.
\end{lem}
\begin{proof}
  For $\Gamma^{\full}$ as above, we have that for every iterate
  $\Gamma^n$,
\begin{equation*}
  [\Gamma^n]=\left(\frac{1}{\deg\nu}\right)^n[(\Gamma^{\full})^n].\qedhere
\end{equation*}
\end{proof}
This means that arbitrary Hurwitz correspondences may be studied via
fully marked Hurwitz spaces.  These in turn have convenient
compactifications by spaces of admissible covers.
\begin{thm}[(Harris and Mumford, \cite{HarrisMumford1982})]
  Given $(\bA,\bB,d,F,\br,\rm)$ satisfying Conditions 1 and $2'$ as in
  Definition \ref{Def:FullyMarked}, there is a projective variety
  $\Hbar=\Hbar(\bA,\bB,d,F,\br,\rm)$ containing
  $\H=\H(\bA,\bB,d,F,\br,\rm)$ as a dense open subset. This
  \emph{admissible covers} compactification $\Hbar$ extends the maps
  $\pi_\bB$ and $\pi_\bA$ to maps $\bar{\pi_{\bB}}$ and
  $\bar{\pi_{\bA}}$ to $\Mbar_{0,\bB}$ and $\Mbar_{0,\bA}$,
  respectively, with $\bar{\pi_\bB}:\Hbar\to\Mbar_{0,\bB}$ a finite
  flat map. $\Hbar$ may not be normal, but its normalization is
  smooth.
\end{thm}

The following comparison of tautological line bundles on moduli spaces of admissible covers is the key ingredient in our proof of Theorem \ref{thm:DynamicalDegreesNonincreasing}:
\begin{prop}[(Ionel, Lemma 1.17 in \cite{Ionel2001})]\label{prop:Ionel}
  Let $\Hbar=\Hbar(\bA,\bB,d,F,\br,\rm)$ be a fully marked space of
  admissible covers with maps $\bar{\pi_\bB}$ and $\bar{\pi_{\bA}}$ to
  $\Mbar_{0,\bB}$ and $\Mbar_{0,\bA}$ respectively. Suppose we have
  $a\in\bA$ and $b\in \bB$ with $F(a)=b.$ Then
  $(\bar{\pi_\bB})^*(\mathcal{L}_b)=(\bar{\pi_{\bA}})^*(\mathcal{L}_a)^{\otimes\rm(a)}$
  as line bundles on $\Hbar$.
\end{prop}

\section{Main Theorem}\label{sec:Proof}

\begin{thm}\label{thm:DynamicalDegreesNonincreasing}
  Let $(\Gamma,\pi_1,\pi_2):\Mbar_{0,\bP}\ratt\Mbar_{0,\bP}$ be a
  dominant Hurwitz self-correspondence. Let $R$ be the polynomiality index of $\Gamma$, and let $\Theta_k$ be the
  $k$th dynamical degree of $\Gamma.$ Then
  \begin{align*}
    \Theta_0\ge R\Theta_1\ge\cdots\ge R^{\abs{\bP}-3}\Theta_{\abs{\bP}-3}.
  \end{align*}
\end{thm}
\begin{proof}
  By Lemma \ref{lem:ReduceToConnectedFullyMarked}, we may assume
  $\Gamma$ is a union of connected components of a fully marked
  Hurwitz space $\H=\H(\Pfull,\bP,d,F,\br,\rm)$ corresponding to a
  superset $\Pfull$ of $\bP$. Let $\Hbar$ denote the admissible covers
  compactification of $\H$, and let $\bar\Gamma$ be the closure of
  $\Gamma$ in $\Hbar.$ For $\ell>0$ set $\Gamma^\ell$ to be the
  $\ell$th iterate of $\Gamma$, that is
  \begin{align*}
    \Gamma\thickspace{_{\pi_2}\times_{\pi_1}}\thickspace\cdots\thickspace{_{\pi_2}\times_{\pi_1}}\thickspace\Gamma\mbox{\quad\quad($\ell$
      times),}
  \end{align*}
  Set $\bar{\Gamma^{\ell}}$ to be its compactification
  \begin{align*}
    \bar\Gamma\thickspace{_{\bar{\pi_2}}\times_{\bar{\pi_1}}}\thickspace\cdots\thickspace{_{\bar{\pi_2}}\times_{\bar{\pi_1}}}\thickspace\bar\Gamma\mbox{\quad\quad($\ell$
      times),}
  \end{align*}
  with $\bar{\pi_1^{\ell}}$ and $\bar{\pi_2^{\ell}}$ its two maps to
  $\Mbar_{0,\bP}$.

  Since $\bar{\pi_1^{\ell}}$ is a flat map, no irreducible component
  of $\bar{\Gamma^\ell}$ is supported over the boundary of
  $\Mbar_{0,\bP}$. This means that $\Gamma^\ell$ is a dense open
  subset of $\bar{\Gamma^\ell}$. We refer to the complement
  $\bar{\Gamma^\ell}\setminus\Gamma^{\ell}$ as the boundary of
  $\bar{\Gamma^\ell}$. The inverse image under $\bar{\pi_1^\ell}$ of
  the boundary of $\Mbar_{0,\bP}$ is exactly the boundary of
  $\bar{\Gamma^\ell}$. The inverse image under $\bar{\pi_2^\ell}$ of
  the boundary of $\Mbar_{0,\bP}$ is contained in the boundary of
  $\bar{\Gamma^\ell}.$

  The compactification $\bar{\Gamma^\ell}$ is singular. However, for
  Cartier divisors $D_1,\ldots,D_{\dim\bar{\Gamma^\ell}},$ the
  intersection product $D_1\cdot\cdots\cdot D_{\dim\bar{\Gamma^\ell}}$
  is a well-defined integer as in Section 1.1.C of
  \cite{Lazarsfeld2004}. For any subscheme $Y$ of dimension $k$, and
  Cartier divisors $D_1,\ldots,D_k,$ we similarly have the
  intersection number $D_1\cdot\cdots\cdot D_k\cdot Y\in\Z.$
  \begin{lem}\label{lem:Cartier}
    For all $p\in\bP$ and for all $\ell\ge0,$ there is an equality of
    Cartier divisors on $\bar{\Gamma^\ell}$ of the form
    \begin{align*}
      (\bar{\pi_1^\ell})^*(\psi_{F^\ell(p)})=\Pi_{i=0}^{\ell-1}\rm(F^i(p))\cdot(\bar{\pi_2^\ell})^*(\psi_{p})+E,
    \end{align*}
    where $E$ is an effective Cartier
    divisor supported on the boundary of $\bar{\Gamma^\ell}.$
  \end{lem}
  \begin{proof}
    We induct on $\ell.$ By convention, $\bar{\Gamma^0}$ is the
    identity rational correspondence
    $$(\Mbar_{0,\bP},\thickspace\bar{\pi_1^0}=\Id,\thickspace\bar{\pi_2^0}=\Id):\Mbar_{0,\bP}\ratt\Mbar_{0,\bP}.$$
    For all $p\in\bP$, $F^0(p)=p$, so
    $(\bar{\pi_1^0})^*(\psi_{F^0(p)})=(\bar{\pi_2^0})^*(\psi_p).$ This
    gives us the base case $\ell=0.$

    Suppose the Lemma holds for $\ell-1.$ We have    
    \begin{center}
      \begin{tikzpicture}
        \matrix(m)[matrix of math nodes,row sep=2.4em,column
        sep=3em,minimum width=2em] {
          &&\bar{\Gamma^\ell}=\bar{\Gamma}\thickspace{_{\bar{\pi_2}}\times_{\bar{\pi_1^{\ell-1}}}}\thickspace\bar{\Gamma^{\ell-1}}&&\\
          &\bar{\Gamma}&&\bar{\Gamma^{\ell-1}}&\\
          &&\Mbar_{0,\Pfull}&&\\
          \Mbar_{0,\bP}&&\Mbar_{0,\bP}&&\Mbar_{0,\bP}\\
        };
        \path[-stealth] (m-1-3) edge node [above left] {$\pr_1$} (m-2-2);
        \path[-stealth] (m-1-3) edge node [above right] {$\pr_2$} (m-2-4);
        \path[-stealth] (m-2-2) edge node [above left] {$\bar{\pi_1}$} (m-4-1);
        \path[-stealth] (m-2-2) edge node [below left] {$\bar{\pi_2}$} (m-4-3);
        \path[-stealth] (m-2-4) edge node [below right] {$\bar{\pi_1^{\ell-1}}$} (m-4-3);
        \path[-stealth] (m-2-4) edge node [above right] {$\bar{\pi_2^{\ell-1}}$}
        (m-4-5);
        \path[-stealth] (m-2-2) edge node [above right]
        {$\bar{\pi_2^{\full}}$} (m-3-3);
\path[-stealth] (m-3-3) edge node [left] {$\mu$} (m-4-3);
      \end{tikzpicture}
    \end{center}
    For all $p\in\bP,$ we have
    \begin{align*}
      (\bar{\pi_1^\ell})^*(\psi_{F^{\ell}(p)})&=\pr_1^*(\bar{\pi_1})^*(\psi_{F^{\ell}(p)})\\
      &=\pr_1^*\left(\rm({F^{\ell-1}(p)})\cdot(\bar{\pi_2^{\mathrm{full}}})^*(\psi_{{F^{\ell-1}(p)}}^{\Pfull})\right)\quad\quad\quad\mbox{(by
        Proposition \ref{prop:Ionel})}.
    \end{align*}
    By Lemma \ref{lem:ForgetPsi},
    \begin{align*}
      \psi_{F^{\ell-1}(p)}^{\Pfull}&=\mu^*(\psi_{F^{\ell-1}(p)})+\sum_{\bS\subseteq\Pfull\setminus\bP}\delta_{\{F^{\ell-1}(p)\}\cup\bS}.
    \end{align*}                                            
    The inverse image under $\bar{\pi_2^{\mathrm{full}}}$ of the
    boundary in $\Mbar_{0,\Pfull}$ is contained in the boundary of
    $\bar{\Gamma}$ (in fact it is the entire boundary), and the
    inverse image under $\pr_1$ of the boundary of $\bar{\Gamma}$ is
    the boundary of $\bar{\Gamma^\ell}.$ Thus, the Cartier divisor
    $$E_1:=\pr_1^*\left((\bar{\pi_2^{\mathrm{full}}})^*\left(\sum_{\mathbf{S}\subseteq\Pfull\setminus\bP}\delta_{\{{F^{\ell-1}(p)}\}\cup
          \mathbf{S}}\right)\right)$$
    is effective and supported on the boundary of $\bar{\Gamma^\ell}.$ We continue:
    \begin{align*}
      (\bar{\pi_1^{\ell}})^*(\psi_{F^{\ell}(p)})&=\rm(F^{\ell-1}(p))\pr_1^*(\bar{\pi_2^{\full}})^*\mu^*(\psi_{F^{\ell-1}(p)})+\rm(F^{\ell-1}(p))E_1\\
      &=\rm(F^{\ell-1}(p))\pr_1^*(\bar{\pi_2})^*(\psi_{F^{\ell-1}(p)})+\rm(F^{\ell-1}(p))E_1\\
      &=\rm(F^{\ell-1}(p))\pr_2^*(\bar{\pi_1^{\ell-1}})^*(\psi_{F^{\ell-1}(p)})+\rm(F^{\ell-1}(p))E_1.
    \end{align*}
    By the inductive hypothesis, we can rewrite this as
    \begin{align*}
      \rm(F^{\ell-1}(p))\pr_2^*(\Pi_{i=0}^{\ell-2}\rm(F^i(p))(\bar{\pi_2^{\ell-1}})^*(\psi_{p})+E_2)+\rm(F^{\ell-1}(p))E_1,
    \end{align*}
    where  $E_2$ is an effective
    Cartier divisor supported on the boundary of
    $\bar{\Gamma^{\ell-1}}.$ Since the inverse image under
    $\pr_2$ of the boundary of $\bar{\Gamma^{\ell-1}}$ is contained in
    the boundary of $\bar{\Gamma^\ell}$, $\pr_2^*(E_2)$ is an
    effective Cartier divisor supported on the boundary of
    $\bar{\Gamma^\ell}.$ Thus we can finally write
    \begin{align*}
     & (\bar{\pi_1^\ell})^*(\psi_{F^{\ell}(p)})\\
     &=\rm(F^{\ell-1}(p))(\Pi_{i=0}^{\ell-2}\rm(F^i(p)))\pr_2^*(\bar{\pi_2^{\ell-1}})^*(\psi_{p})+\rm(F^{\ell-1}(p))\pr_2^*(E_2)+\rm(F^{\ell-1}(p))E_1\\
      &=\Pi_{i=0}^{\ell-1}\rm(F^i(p))(\bar{\pi_2^{\ell}})^*(\psi_{p})+(\rm(F^{\ell-1}(p))\pr_2^*(E_2)+\rm(F^{\ell-1}(p))E_1),
    \end{align*}
    which is as desired. This proves Lemma \ref{lem:Cartier}.
  \end{proof}  
  Now, since $F:\bP\to\bP$ is a map of finite sets, every point is
  eventually periodic. Fix $p\in\bP$ that is periodic of period $\ell_{0}>0$ and such that $(\Pi_{i=0}^{\ell_{0}-1}\rm(F^i(p)))^{\frac{1}{\ell_{0}}}=R.$ Then by Lemma \ref{lem:Cartier}, for every multiple $m\ell_{0},$ we have on
  $\bar{\Gamma^{m\ell_{0}}}$:
  \begin{align}\label{eq:ComparePsi}
    (\bar{\pi_1^{m\ell_{0}}})^*(\psi_{p})=R^{m\ell_{0}}(\bar{\pi_2^{m\ell_{0}}})^*(\psi_p)+E_m,
  \end{align}
  where $E_m$ is an effective Cartier
  divisor supported on the boundary of $\bar{\Gamma^{m\ell_{0}}}.$
  
  Let $\rho:\Mbar_{0,\bP}\to\C\P^{\abs{\bP}-3}$ be the
  birational morphism to projective space given by the line bundle
  $\mathcal{L}_p.$ Let $\mathfrak{h}$ be the Cartier divisor class of
  a hyperplane in $\C\P^{\abs{\bP}-3}.$ Then $\rho^*(\mathfrak{h})=\psi_p.$

  The pullback $[\Gamma^n]^*(\mathfrak{h}^k)$ is by
  definition $$(\rho\circ\bar{\pi_1^n})_*\circ(\rho\circ\bar{\pi_2^n})^*(\mathfrak{h}^k).$$
  So, by the projection formula,
    \begin{align*}
      ([\Gamma^n]^*(\mathfrak{h}^k))\cdot(\mathfrak{h}^{\abs{\bP}-3-k})&=((\rho\circ\bar{\pi_2^n})^*(\mathfrak{h}^k))\cdot((\rho\circ\bar{\pi_1^n})^*(\mathfrak{h}^{\abs{\bP}-3-k})).
  \end{align*}  
  Since dynamical degrees are birational invariants, $\Theta_k$ is
  also the $k$th dynamical degree of the induced rational
  correspondence
  $(\Gamma,\rho\circ\bar{\pi_1},\rho\circ\bar{\pi_2}):\C\P^{\abs{\bP}-3}\ratt\C\P^{\abs{\bP}-3}.$
  We have
  \begin{align*}
    \Theta_k&=\lim_{n\to\infty}(([\Gamma^n]^*(\mathfrak{h}^k))\cdot(\mathfrak{h}^{\abs{\bP}-3-k}))^{1/n}\\
    &=\lim_{n\to\infty}(((\rho\circ\bar{\pi_2^n})^*(\mathfrak{h}^k))\cdot((\rho\circ\bar{\pi_1^n})^*(\mathfrak{h}^{\abs{\bP}-3-k})))^{1/n}\\
    &=\lim_{n\to\infty}(((\bar{\pi_2^n})^*(\psi_p^k))\cdot((\bar{\pi_1^n})^*(\psi_p^{\abs{\bP}-3-k})))^{1/n}.
  \end{align*}
  Since this sequence converges, we can find its limit using any
  subsequence,
  and \begin{align*}
  \Theta_k&=\lim_{m\to\infty}(((\bar{\pi_2^{m\ell_{0}}})^*(\psi_p^k))\cdot((\bar{\pi_1^{m\ell_{0}}})^*(\psi_p^{\abs{\bP}-3-k})))^{1/m\ell_{0}}\\
  &=\lim_{m\to\infty}(((\bar{\pi_2^{m\ell_{0}}})^*(\psi_p))^k\cdot((\bar{\pi_1^{m\ell_{0}}})^*(\psi_p))^{\abs{\bP}-3-k})^{1/m\ell_{0}}.
  \end{align*}  
  For $m>0$, set 
  \begin{align*}
      \alpha_{m,k}:=((\bar{\pi_2^{m\ell_{0}}})^*(\psi_p))^k\cdot((\bar{\pi_1^{m\ell_{0}}})^*(\psi_p))^{\abs{\bP}-3-k},
    \end{align*}
    so 
    \begin{align*}
    \Theta_k&=\lim_{m\to\infty}  (\alpha_{m,k})^{1/m\ell_{0}}
    \end{align*}
  \begin{lem}\label{lem:IntersectionNumbers}
    Fix $m>0$. The intersection numbers $\alpha_{m,k}$ on $\bar{\Gamma^{m\ell_{0}}}$ satisfy
    \begin{align*}
    \alpha_{m,0}\ge R^{m\ell_{0}}\alpha_{m,1}\ge\cdots\ge (R^{m\ell_{0}})^{\abs{\bP}-3}\alpha_{m,\abs{\bP}-3}
    \end{align*}
  \end{lem}
  \begin{proof}[Proof of Lemma \ref{lem:IntersectionNumbers}]
      
    Let $\bar{\cJ}$ be any irreducible component of $\bar{\Gamma^{m\ell_{0}}}$, and set
        \begin{align*}
      \alpha_{\bar{\mathcal{J}},k}:=((\bar{\pi_2^{m\ell_{0}}})^*(\psi_p))|_{\bar{\mathcal{J}}}^k\cdot((\bar{\pi_1^{m\ell_{0}}})^*(\psi_p))|_{\bar{\mathcal{J}}}^{\abs{\bP}-3-k}.
    \end{align*}

Since $(\bar{\pi_1^{m\ell_{0}}})^*(\psi_p)$ and
    $(\bar{\pi_2^{m\ell_{0}}})^*(\psi_p)$ are pullbacks of the ample
    hyperplane class $\mathfrak{h}$, they are nef on
    $\bar{\Gamma^{m\ell_{0}}}$ and $\bar{\cJ}$. By \cite{Lazarsfeld2004}, Example 1.6.4,
    $\alpha_{\bar{\mathcal{J}},k}$ is a log-concave function of $k$. 
    
     Note that $\psi_p^{\abs{\bP}-4}=\rho^*(\mathfrak{h}^{\abs{\bP}-4}).$ The class
    $\mathfrak{h}^{\abs{\bP}-4}$ on $\C\emph\P^{\abs{\bP}-3}$ may be represented by a line
    $L$ that does not intersect the codimension-two exceptional locus
    of $\rho.$ Then $\rho^{-1}(L)$ is an irreducible curve in
    $\Mbar_{0,\bP}$ not contained in the boundary and
    $(\bar{\pi_1^{m\ell_{0}}})^{-1}(\rho^{-1}(L))|_{\bar{\cJ}}$ is a curve
    $Y$ none of whose irreducible components lies in the boundary of
    $\bar{\cJ}.$ Since $\bar{\pi_1^{m\ell_{0}}}$ is a flat map, and a
    covering map away from the boundary,
    \begin{align*}
      ((\bar{\pi_1^{m\ell_{0}}})^*(\psi_p^{\abs{\bP}-4}))|_{\bar{\mathcal{J}}}=[Y].
    \end{align*}
By Equation \ref{eq:ComparePsi}, we have
  \begin{align*}
    (\bar{\pi_1^{m\ell_{0}}})^*(\psi_{p})\cdot
    [Y]&=R^{m\ell_{0}}(\bar{\pi_2^{m\ell_{0}}})^*(\psi_p)\cdot [Y]+E_m\cdot [Y]
  \end{align*}
  Since $(\bar{\pi_1^{m\ell_{0}}})^*(\psi_p)$ and
  $(\bar{\pi_2^{m\ell_{0}}})^*(\psi_p)$ are nef on $\bar{\Gamma^{m\ell_{0}}},$ the intersection numbers $(\bar{\pi_1^{m\ell_{0}}})^*(\psi_{p})\cdot
    [Y]$ and $(\bar{\pi_2^{m\ell_{0}}})^*(\psi_p)\cdot [Y]$ are non-negative. Since $E_m$ is entirely supported on the boundary and no component of $Y$ is supported on the boundary, $E_m\cdot [Y]$ is non-negative as well. Thus we obtain:
  \begin{align}\label{eqn:Ineq}
    (\bar{\pi_1^{m\ell_{0}}})^*(\psi_{p})\cdot
    [Y]&\ge R^{m\ell_{0}}(\bar{\pi_2^{m\ell_{0}}})^*(\psi_p)\cdot [Y].
  \end{align}  

       
        Thus,
    \begin{align*}
      \alpha_{\bar{\mathcal{J}},0}\\&=((\bar{\pi_1^{m\ell_{0}}})^*(\psi_p))|_{\bar{\cJ}}^{\abs{\bP}-3}\\&=((\bar{\pi_1^{m\ell_{0}}})^*(\psi_p))|_{\bar{\cJ}}\cdot((\bar{\pi_1^{m\ell_{0}}})^*(\psi_p^{\abs{\bP}-4}))|_{\bar{\cJ}}\\
      &=((\bar{\pi_1^{m\ell_{0}}})^*(\psi_p))|_{\bar{\cJ}}\cdot [Y]\\
      &\ge R^{m\ell_{0}}((\bar{\pi_2^{m\ell_{0}}})^*(\psi_p))|_{\bar{\cJ}}\cdot
      [Y]\quad\quad\quad\quad\quad\quad\quad\mbox{by \eqref{eqn:Ineq}}\\
      &=R^{m\ell_{0}}((\bar{\pi_2^{m\ell_{0}}})^*(\psi_p))|_{\bar{\cJ}}\cdot((\bar{\pi_1^{m\ell_{0}}})^*(\psi_p^{\abs{\bP}-4}))|_{\bar{\cJ}}\\
      &=R^{m\ell_{0}}((\bar{\pi_2^{m\ell_{0}}})^*(\psi_p))|_{\bar{\cJ}}\cdot((\bar{\pi_1^{m\ell_{0}}})^*(\psi_p))|_{\bar{\cJ}}^{\abs{\bP}-4}\\&=R^{m\ell_{0}} \alpha_{\bar{\mathcal{J}},1}.
    \end{align*}
    By log-concavity, we conclude that the intersection numbers $\alpha_{\bar{\mathcal{J}},k}$ satisfy
    \begin{align*}
    \alpha_{\bar{\mathcal{J}},0}\ge R^{m\ell_{0}}\alpha_{\bar{\mathcal{J}},1}\ge\cdots\ge (R^{m\ell_{0}})^{\abs{\bP}-3}\alpha_{\bar{\mathcal{J}},\abs{\bP}-3}
    \end{align*}
    
    Since $$\alpha_{m,k}=\sum_{\substack{\bar{\mathcal{J}}\text{ irreducible}\\\text{component of }\bar{\Gamma^{m\ell_{0}}}}}\alpha_{\bar{\mathcal{J}},k},$$
    the lemma follows.
    
      \end{proof}  
  We now complete the proof of Theorem \ref{thm:DynamicalDegreesNonincreasing}. For all $m$,
  \begin{align*}
    &\alpha_{m,0}\ge R^{m\ell_{0}}\alpha_{m,1}\ge\cdots\ge (R^{m\ell_{0}})^{\abs{\bP}-3}\alpha_{m,\abs{\bP}-3}, \quad\quad\text{so}\\
     &\alpha_{m,0}^{1/m\ell_0}\ge R \alpha_{m,1}^{1/m\ell_0}\ge\cdots\ge R^{\abs{\bP}-3}\alpha_{m,\abs{\bP}-3}^{1/m\ell_0}
    \end{align*}
 The theorem follows by taking the limit as $m$ goes to infinity.
\end{proof}

\section{An application to enumerative algebraic geometry}\label{sec:enumerative}

\begin{prop}\label{prop:enumerative}
Let $\H=\H(\bP,\bP,d,F,\br,\rm)$ be a Hurwitz space with ``target" and ``source" maps $\pi_1$ and $\pi_2$ respectively to $\M_{0,\bP}$. Let 
$$r=\max_{p\in\bP}\rm(p).$$ 
Let $\Gamma$ be any connected component of $\H$. Then $\deg(\pi_1|_{\Gamma})\ge r^{\abs{\bP}-3} \deg(\pi_2|_{\Gamma})$. 
\end{prop}
\begin{rem}
Here, $r$ is the maximum local degree of $[f:\C\P^1\to\C\P^1]\in\H$ at $p$ where $p$ ranges over $\bP$. Note that every connected component $\Gamma$ arises at $\H_{\phi}$ for some topological branched covering $\phi$. Proposition \ref{prop:enumerative} states that the number of ways a generic configuration of $\bP$-marked points on $\C\P^1$ arises the configuration of marked points on the target $\C\P^1$ of some branched map $[f]\in\H$ of fixed topological type is at least $r^{\abs{\bP}-3}$ times the number of ways it appears as the configuration of marked points on the source $\C\P^1$ of such a map. Since the same inequality holds for every connected component of $\H$/topological type of branched covering, we obtain by summing over connected components that $\deg(\pi_1)\ge r^{\abs{\bP}-3} \deg(\pi_2)$. That is, the number of ways a generic configuration of $\bP$-marked points on $\C\P^1$ arises the configuration of marked points on the target $\C\P^1$ of some branched map $[f]\in\H$ (without fixing topological type) is at least $r^{\abs{\bP}-3}$ times the number of ways it appears as the configuration of marked points on the source $\C\P^1$ of such a map.
\end{rem}
\begin{proof}
Fix $p\in\bP$ with $\rm(p)=r$. Pick a permutation $\sigma\in\Aut(\bP)$ such that $\sigma(p)=F(p)$. The permutation $\sigma$ induces an automorphism $\sigma^{\moduli}$ of $\M_{0,\bP}$ given by $$[\iota:\bP\into\C\P^1]\mapsto[\iota\circ\sigma:\bP\into\C\P^1].$$ 

Set $$\H^{\new}=\H(\bP,\bP,d,\sigma^{-1}\circ F, \br\circ\sigma,\rm).$$ $\H^{\new}$ is the Hurwitz space obtained by using $\sigma$ to relabel the marked points of $[f:\C\P^1\to\C\P^1]\in\H$ on the target $\C\P^1$. The point $p$ is a `fixed point' of maps $[f]\in\H^{\new}$. Note that by construction $\PI(\H^{\new})\ge r$. (In fact $\PI(\H^{\new})= r$, although we will not need this stronger fact.) There is an isomorphism $\sigma^{\hurwitz}$ from $\H$ to $\H^{\new}$ as follows. A point in $\H$ is a map $f:\C\P^1\to\C\P^1$ together with injections $\iota_1$ and $\iota_2$ into the target $\C\P^1$ and source $\C\P^1$ respectively. The isomorphism $\sigma^{\hurwitz}:\H\to\H^{\new}$ takes $[f,\iota_1,\iota_2]\in\H$ to $[f,\iota_1\circ\sigma,\iota_2]$. Denote by $\pi^{\new}_1$ and $\pi^{\new}_2$ respectively the ``target" and ``source" maps from $\H^{\new}$ to $\M_{0,\bP}$. Note that $\sigma^{\moduli}\circ\pi_1=\pi_1^{\new}\circ\sigma^{\hurwitz}$; also that $\pi_2=\pi_2^{\new}\circ\sigma^{\hurwitz}.$ 

Now let $\Gamma$ be some connected component of $\H$; denote by $\Gamma^{\new}$ its isomorphic image in $\H^{\new}$. 
By Theorem \ref{thm:DynamicalDegreesNonincreasing}, we have
\begin{align}\label{eq:twodegrees1}
\deg(\pi_1^{\new}|_{\Gamma^{\new}})=\Theta_0(\Gamma^{\new})\ge r^{\abs{\bP}-3}\Theta_{\abs{bP}-3}(\Gamma^{\new})=\deg(\pi_2^{\new}|_{\Gamma^{\new}})
\end{align} 

Since $\sigma^{\moduli}$ and $\sigma^{\hurwitz}$ are both isomorphisms, we have that $$\deg(\pi_1|_{\Gamma})=\deg(\pi_1^{\new}|_{\Gamma^{\new}})$$ and $$\deg(\pi_2|_{\Gamma})=\deg(\pi_2^{\new}|_{\Gamma^{\new}}).$$

By (\ref{eq:twodegrees1}), this proves the proposition.

\end{proof}
\bibliographystyle{amsalpha}
\bibliography{../HurwitzRefs}
\end{document}